\documentclass{amsart}

\usepackage{amssymb,amsmath,tikz}
\usetikzlibrary{positioning}
\usetikzlibrary{calc,through,backgrounds}
\usetikzlibrary{decorations.markings}
\usepackage[bookmarks=true,%
    colorlinks=true,%
    linkcolor=blue,%
    citecolor=blue,%
    filecolor=blue,%
    menucolor=blue,%
    urlcolor=blue,%
    breaklinks=true]{hyperref}

\newcommand{\Z}{\mathbb{Z}}

\newcommand{\R}{\mathbb{R}}

\newcommand{\T}{\mathbb{T}}

\newcommand{\gau}[1]{\langle #1 \rangle}
\newcommand{\mfun}[1]{\operatorname{f}_{#1}}
\newcommand{\minv}[1]{\alpha(#1)}
\newcommand{\mbe}[1]{\operatorname{S}_{#1}}
\newcommand{\me}[1]{\operatorname{\Lambda}(#1)}
\newcommand{\ms}[1]{\operatorname{\sigma}(#1)}
\newtheorem{theorem}{Theorem}
\newtheorem{lemma}{Lemma}

\theoremstyle{definition}

\newtheorem{remark}{Remark}

\title{A dequantized metaplectic knot invariant}
\author{Rinat Kashaev}
\address{Section de math\'ematiques, Universit\'e de Gen\`eve,
2-4 rue du Li\`evre, 1211 Gen\`eve 4, Suisse\\}
\email{rinat.kashaev@unige.ch}
\date{April 24, 2017}

\thanks{Supported in part by Swiss National Science Foundation}
 \dedicatory{In memory of Ludwig Faddeev}
\begin{document}
\begin{abstract} 
Let $K\subset S^3$ be a knot, $X:= S^3\setminus K$ its complement, and $\T$ the circle group identified with $\R/\Z$. To any oriented long knot diagram of $K$, we associate a quadratic polynomial  in variables bijectively associated with the bridges of the diagram such that, when the variables projected to $\T$ satisfy the linear equations characterizing the first homology group $H_1(\tilde{X}_2)$ of the double cyclic covering of $X$, the polynomial projects down to a well defined $\T$-valued function on  $T^1(\tilde{X}_2,\T)$ (the dual of the torsion part $T_1$ of $H_1$). This function is sensitive to knot chirality, for example, it seems to confirm  chirality of the knot $10_{71}$. It also distinguishes the knots $7_4$ and $9_2$ known to have identical Alexander polynomials and the knots $9_2$ and K11n13 known to have identical Jones polynomials but does not distinguish $7_4$ and K11n13.
\end{abstract}
\maketitle
\section{Introduction}
The \emph{metaplectic} knot invariants were originally introduced in \cite{MR974081} under the name of \emph{cyclotomic} invariants and subsequently generalised in \cite{MR990215,MR1012438} where the name ``metaplectic'' has also been introduced. These are quantum invariants specified by a choice of a finite field or, more generally, a finite cyclic group, but without specifying  a deformation parameter like $q$ or Planck's constant. We think that the term ``metaplectic'' is  more appropriate because those invariants can be generalized to the context of arbitrary locally compact abelian groups which are self-dual in the sense of Pontryagin and admit a \emph{gaussian exponential}, a special case of Weil's \emph{second order character}~\cite{MR0165033}.  We call such groups \emph{gaussian}. A simple and prototypical example of an infinite gaussian group is the real line $\R$ with the gaussian exponential 
\begin{equation}
\gau{x}:=e^{\pi i x^2},\quad \forall x\in \R.
\end{equation}
The relevance of gaussian exponentials for quantum topology comes from the fact that they satisfy the (constant) Yang--Baxter relation in the form called \emph{star-triangle  relation with difference property}~\cite{MR690578,MR990215}, which equivalently can be formulated as an identity for unitary operators in the Hilbert space of square integrable functions on the underlying gaussian group~\cite{MR3491319}. In the case of $\R$ this takes the form
\begin{equation}
e^{\pi i \hat p^2}e^{\pi i \hat q^2}e^{\pi i \hat p^2}=e^{\pi i \hat q^2}e^{\pi i \hat p^2}e^{\pi i \hat q^2}
\end{equation}
where $\hat p$ and $\hat q$ are  self-adjoint Heisenberg's \emph{momentum} and \emph{position} operators in $L^2(\R)$ satisfying the commutation relation
\begin{equation}
\hat p\hat q-\hat q\hat p=\frac{1}{2\pi i}
\end{equation}
which corresponds to the normalisation where Planck's constant $\hbar$ is fixed by the equality $2\pi\hbar=1$. Interestingly, even in the case of infinite gaussian groups, these solutions  can be utilised for construction of invariants of at least long knots or, more generally, string links, without encountering the problem of divergencies. In the case of the gaussian group $\R$, the associated invariant is conjecturally equal to
\begin{equation}
\frac{e^{\frac{\pi i}{4}\text{signature}}}{\sqrt{\text{determinant}}}
\end{equation}
where the fact that the phase factor is  an integer multiple of $\pi/4$ is due to the known gaussian integral
\begin{equation}
\int_{\R}e^{\pi i x^2}\operatorname{d}\! x=e^{\frac{\pi i}4}.
\end{equation}
In this context, the term ``metaplectic invariant'' is more relevant than ``cyclotomic invariant'' because the metaplectic extension of $SL(2,\R)$ in quantum mechanics is generated by operator valued gaussian exponentials of the form
\begin{equation}
e^{\pi i \hat p^2}, \ e^{\pi i x\hat q^2},\quad \forall x\in \R.
\end{equation}

In this paper, we ``dequantize'' this example in the sense that we write it as a sum of the form
\begin{equation}
\frac{1}{\text{determinant}}\sum_{x\in T^1(\tilde X_2,\T)}e^{2\pi i \mfun{}(x)},
\end{equation}
where $T^1(\tilde X_2,\T)$ is the dual of the torsion part of the first homology of the double cyclic covering space $\tilde X_2$ of the knot complement $X$ with coefficients in the circle group $\T\simeq\R/\Z$, and $\mfun{}(x)$ is a certain function invariantly associated to the knot.  Among the properties of $\mfun{}(x)$, it is worth mentioning that it changes its sign  for the mirror image of the knot, though, at the moment of this writing I do not know how to prove this formally.

To illustrate how this function looks like in actual calculations, consider
a knot $K$ with $T_1(\tilde X_2)\simeq\Z/n\Z$, so that we can identify
\begin{equation}
T^1(\tilde X_2,\T)=\{x\in \T\ \vert\ nx=0\}.
\end{equation}
In the examples we have analysed, the invariant function always takes the form
\begin{equation}
\mfun{}(x)=n\alpha x^2, \quad \alpha\in\Z.
\end{equation}
As we do not have any canonical choice for the generator of $T_1(\tilde X_2)$, one way of extracting an invariant information would be to consider the whole orbit of $\mfun{}(x)$ under the action of the automorphism group of $T_1(\tilde X_2)$, i.e. the set of functions of the form
\begin{equation}
x\mapsto n \alpha k^2 x^2,\quad k\in \operatorname{Aut}(\Z/n\Z)=\{k\in\Z_{>0}\cap\Z_{<n}\ \vert\ \operatorname{gcd}(k,n)=1\}.
\end{equation}
Now, the subset $\minv{K}$ of elements of the smallest absolute value in the set
\begin{equation}
\bigcup_{k\in \operatorname{Aut}(\Z/n\Z)}\left\{ m\in \Z\ \left\vert\ |m|\le \frac{n-1}2,\ \alpha=mk^2\pmod n\right.\right\}
\end{equation}
 is an invariant of $K$. It is either a non-zero singleton, which will mean that the knot is chiral under our hypothesis that the invariant functions for the mirror pairs differ by sign, or it is  the zero singleton (this is certainly the case when $T_1(\tilde X_2)=0$) or else a pair of non-zero numbers of opposite sign, and in these latter two cases the invariant does not distinguish the knot from its mirror image.
 \subsection{The trefoil knot}
 Taking for $3_1$ the left-handed trefoil (see the text below where it is used as an illustrative example), we have
\begin{equation}
T_1(\tilde{X}_2)\simeq\Z/3\Z,\quad X=S^3\setminus 3_1,
\end{equation}
and
\begin{equation}
\minv{3_1}=\{-1\}
\end{equation}
thus confirming chirality of the trefoil.
\subsection{The figure-eight knot}
We have
\begin{equation}
T_1(\tilde{X}_2)\simeq\Z/5\Z,\quad X=S^3\setminus 4_1,
\end{equation}
and
\begin{equation}
\minv{4_1}=\{\pm 1\}
\end{equation}
which is consistent with amphichirality of the figure-eight knot.
\subsection{The knot $6_1$}
We have
\begin{equation}
T_1(\tilde{X}_2)\simeq\Z/9\Z,\quad X=S^3\setminus 6_1,
\end{equation}
and
\begin{equation}
\minv{6_1}=\{1\}
\end{equation}
which seemingly distinguishes $6_1$ from its mirror image despite the fact that its signature vanishes.

\subsection{The knot $10_{71}$}
We have
\begin{equation}
T_1(\tilde{X}_2)\simeq\Z/77\Z,\quad X=S^3\setminus 10_{71},
\end{equation}
and
\begin{equation}
\minv{10_{71}}=\{-1\}
\end{equation}
seemingly confirming chirality of $10_{71}$ despite the fact that many known invariants fail to do so.

\subsection{The knots $7_4$, $9_2$, and K11n13}
We have
\begin{equation}
T_1(\tilde{X}_2)\simeq\Z/15\Z,\quad X=S^3\setminus K,\quad K\in\left\{7_4,9_2,\text{K11n13}\right\},
\end{equation}
and 
\begin{equation}
\minv{7_4}= \minv{\text{K11n13}}=\{-2\}
\end{equation}
while
\begin{equation}
\minv{9_2}=\{1\}
\end{equation}
thus distinguishing between $7_4$ and $9_2$ known to have identical Alexander polynomials and also between $9_2$ and K11n13 known to have identical Jones polynomials but not distinguishing between $7_4$ and K11n13.

\subsection{The knots $8_8$, $10_{129}$}
We have
\begin{equation}
T_1(\tilde{X}_2)\simeq\\Z/25\Z,\quad X=S^3\setminus K,\quad K\in\left\{8_8,10_{129}\right\},
\end{equation}
and 
\begin{equation}
 \minv{8_8}=\{\pm2\},\quad  \minv{10_{129}}=\{\pm1\}
\end{equation}
thus distinguishing between these two knots which are known to have coinciding Alexander and Jones polynomials as well as Khovanov homology~\cite{MR2350287}. However, the chiralities of both knots remain undetected. 
\subsection{The knot $9_{46}$}
This is an example of a non-cyclic homology $T_1(\tilde X_2)\simeq (\Z/3\Z)^2$, so that we have the identification
\begin{equation}
T^1(\tilde X_2,\T)=\{(x,y)\in\T^2\ \vert\ 3x=3y=0\},\quad X=S^3\setminus 9_{46},
\end{equation}
and the invariant function takes the form
\begin{equation}
\mfun{}(x,y)=3xy.
\end{equation}
which does not look to contain the information about the chirality of the knot.

In the next section we describe in detail the construction of this function for arbirary long knot diagrams and formulate the main Theorem~\ref{thm}. The proof of this theorem is quite technical, uses the machinery of quantum integrable models adapted to a distribution valued  IRF model with gauge invariance along the line described in~\cite{MR3491319} and vertex-face transformations. These techniques are not very essential for the calculation of the invariant function according to the rules we describe here. For that reason, we leave the proof for a later writing.  It is also possible that there exists a direct combinatorial proof, without using the techniques of quantum integrable models but at the moment of this writing I am not aware of such a proof. 

\subsection*{Acknowledgements} This work is supported in part by Swiss National Science Foundation.

\section{Formulation of the result}
Let $D$ be a long (polygonal) knot diagram drawn horizontally. By convention, the diagram is always oriented from right to left as in this picture:
$$
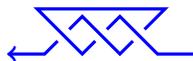
\begin{tikzpicture}[baseline=5,scale=0.5]
\coordinate (v0)  at (0,0);
\coordinate
(v1) at (1,0);
\coordinate 
(v2)  at (2,1);
\coordinate 
(v3)  at (3,0);
\coordinate (v4)  at (4.2,1.2);
\coordinate   (v5)  at (0.8,1.2);
\coordinate  
(v6)  at (2,0);
\coordinate  
(v7)  at (3,1);
\coordinate  
(v8)  at (4,0);
\coordinate (d1)  at  ($(v1)!.75!(v2)$);
\coordinate (d2)  at  ($(v6)!.75!(v7)$);
\coordinate (d3)  at  ($(v3)!.625!(v4)$);

\coordinate (c1)  at (intersection of v1--v2 and v5--v6);
\node [inner sep=2pt] (c1n)  at (c1){};
\coordinate (c2)  at (intersection of v2--v3 and v6--v7);
\node [inner sep=2pt] (c2n)  at (c2){};
\coordinate (c3)  at (intersection of v3--v4 and v7--v8);
\node [inner sep=2pt] (c3n)  at (c3){};
\coordinate (v9)  at (5,0);
    \draw[->,very thick,blue] (v9)--(v8)--(v7)--(c2n)--(v6)--(v5)--(v4)--(c3n)--(v3)--(v2)--(c1n)--(v1)--(v0);
    
\end{tikzpicture}
$$
 where we depict a long knot diagram of the left handed trefoil.
A \emph{bridge} of $D$ is an over-passing path along $D$ between either two under-crossings or an under-crossing and one of the two infinities. By using the orientation of $D$, each bridge has well defined \emph{beginning} and  \emph{ending} points (including the points at infinity for the infinite bridges).  We let $C_D$ and $B_D$ denote the sets of all crossings and all bridges of $D$ respectively. The total number of bridges is always the number of crossings plus one. 

A \emph{crossing vertex} is a point on a bridge exactly at an over-passing crossing.  On each bridge, with the exception of the left infinite one, we mark a generic point in the neighborhood of its ending point and call it \emph{bridge vertex}. The condition to be generic here means that there are no pairs of bridge vertices or bridge and crossing vertices aligned vertically. 
We complement the diagram $D$ by adding auxiliary vertical half-lines that start at bridge vertices and go downwards intersecting $D$ in finitely many points which we call \emph{intersections}, and we enumerate them counting from below upwards with a separate enumeration for each line, and the enumerations on different lines are independent of each other. 
In the following picture, the crossing and bridge vertices are drawn as small and large filled circles respectively and the auxiliary lines as thin black lines:
$$
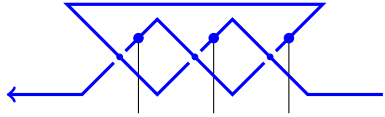
\begin{tikzpicture}[baseline=5,scale=1]
\coordinate (v0)  at (0,0);
\coordinate
(v1) at (1,0);
\coordinate 
(v2)  at (2,1);
\coordinate 
(v3)  at (3,0);
\coordinate (v4)  at (4.2,1.2);
\coordinate   (v5)  at (0.8,1.2);
\coordinate  
(v6)  at (2,0);
\coordinate  
(v7)  at (3,1);
\coordinate  
(v8)  at (4,0);
\coordinate (d1)  at  ($(v1)!.75!(v2)$);
\coordinate (d2)  at  ($(v6)!.75!(v7)$);
\coordinate (d3)  at  ($(v3)!.625!(v4)$);

\coordinate (c1)  at (intersection of v1--v2 and v5--v6);
\node [inner sep=2pt] (c1n)  at (c1){};
\coordinate (c2)  at (intersection of v2--v3 and v6--v7);
\node [inner sep=2pt] (c2n)  at (c2){};
\coordinate (c3)  at (intersection of v3--v4 and v7--v8);
\node [inner sep=2pt] (c3n)  at (c3){};
\coordinate (v9)  at (5,0);
\draw (d1)--+(-90:1) 
(d2)--+(-90:1) 
(d3)--+(-90:1) 
;
    \draw[->,very thick,blue] (v9)--(v8)--(v7)--(c2n)--(v6)--(v5)--(v4)--(c3n)--(v3)--(v2)--(c1n)--(v1)--(v0);
    
\draw[scale=1.25,blue,fill=blue] (d1) circle (.05) (d2) circle (.05) (d3) circle (.05)(c1) circle (.03) (c2) circle (.03) (c3) circle (.03);
\end{tikzpicture}
$$
 The bridge vertices themselves are treated as intersections depending on which side  of the local part of the diagram the vertical line departs from. Namely, a bridge vertex is treated as an intersection if the line departs to the right from the direction of the diagram as in this picture: 
\begin{tikzpicture}[baseline=-3]
\coordinate (c)  at (0,0);
\node [inner sep=3pt] (cn)  at (c){};
\draw (-.4,0)--+(-90:.3);
\draw[very thick,blue]  (0,-.2)--
(c)--
  (0,0.2);
   \draw[->,very thick,blue]  (-1,0)-- (cn)--(1,0);
   \draw[scale=1,blue,fill=blue] 
(c) circle (.04) (-.4,0) circle (.08) ;
\end{tikzpicture}, and it is not treated as an intersection if the vertical line departs to the left  as in this picture:
\begin{tikzpicture}[baseline=-3]
\coordinate (c)  at (0,0);
\node [inner sep=3pt] (cn)  at (c){};
\draw (.4,0)--+(-90:.3);
\draw[very thick,blue]  (0,-.2)--
(c)--
  (0,0.2);
   \draw[->,very thick,blue]  (1,0)-- (cn)--(-1,0);
   \draw[scale=1,blue,fill=blue] 
(c) circle (.04) (.4,0) circle (.08) ;
\end{tikzpicture}.

A \emph{vertex} is either a crossing vertex or a bridge vertex or the beginning point of a bridge or else the point at left infinity \footnote{This is in order to include the ending point of the left infinite bridge where no bridge vertex has been specified, while the right infinity is already a vertex as the beginning point of the right infinite bridge.}.
A \emph{(bridge) segment} is a part of a bridge between any (distinct) vertices. So the set of all bridge segments is in bijection with the set of edges of the underlying four-valent graph of $D$ with both infinite edges included. The orientation of $D$ induces a linear order on the set of segments both globally and on each bridge separately so that, in particular, it makes sense to talk about the first and the last segments among all segments or the first and the last segments of a given bridge.

Let $A_D:=\Z[B_D,\epsilon]$ be the polynomial algebra with integer coefficients with elements of  the set $B_D$ as indeterminates plus one additional element $\epsilon$ which is inverse of 2, i.e. it satisfies the relation $2\epsilon=1$. 

An assignment of elements of $A_D$ to segments of $D$ is called \emph{admissible} if it satisfies the following rules:
\begin{enumerate}
\item \label{rule1} if a segment belongs to a bridge $x$ then the assigned element is of the form $x+m\epsilon$, where $m\in \Z$;
 \item \label{rule2} if  for two segments sharing a common crossing vertex the assigned elements are $\alpha$ and $\beta$, and they are arranged as in this picture 
 $$ 
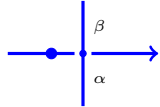
\begin{tikzpicture}
\coordinate (c)  at (0,0);
\node [inner sep=3pt] (cn)  at (c){};
  \draw[very thick,blue]  (0,-.7)-- node [right]{\textcolor{black}{\tiny $\alpha$}}(c)--node [right]{\textcolor{black}{\tiny $\beta$}}(0,0.7);
   \draw[->,very thick,blue]  (-1,0)-- (cn)--(1,0);
   \draw[scale=1.4,blue,fill=blue] 
(c) circle (.03) (-.3,0) circle (.05) ;
\end{tikzpicture}
 $$
 then $\beta=\alpha+\epsilon$ (independently of the orientation of the over-passing part of $D$);
\item \label{rule3}for any two distinct bridges $x$ and $y$ sharing a common under-pass, i.e. there is a crossing where $x$ ends and $y$ starts or vice versa, and if the elements assigned to the last segment of $x$ and the first segment of $y$ (or vice versa)  are $x+m\epsilon$ and $y+n\epsilon$, then $m+n$ is an odd integer;
\item\label{rule4} the elements assigned to two infinite segments are of the form $x+m\epsilon$ with $m=0$ and $x\in B_D$.
\end{enumerate}
The following picture illustrates an admissible assignment in the case of a left handed trefoil diagram:
$$
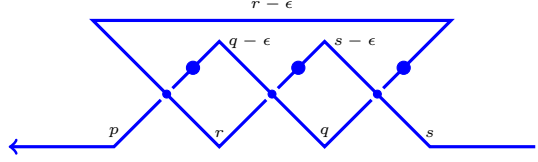
\begin{tikzpicture}[baseline=0,scale=1.4]
\coordinate (v0)  at (0,0);
\coordinate [label=above:\textcolor{black}{\tiny $p$}]  (v1) at (1,0);
\coordinate [label=right:\textcolor{black}{\tiny $q-\epsilon$}] (v2)  at (2,1);
\coordinate [label=above:\textcolor{black}{\tiny $q$}] (v3)  at (3,0);
\coordinate (v4)  at (4.2,1.2);
\coordinate   (v5)  at (0.8,1.2);
\coordinate  [label=above:\textcolor{black}{\tiny $r$}] (v6)  at (2,0);
\coordinate  [label=right:\textcolor{black}{\tiny $s-\epsilon$}] (v7)  at (3,1);
\coordinate  [label=above:\textcolor{black}{\tiny $s$}] (v8)  at (4,0);
\coordinate (d1)  at  ($(v1)!.75!(v2)$);
\coordinate (d2)  at  ($(v6)!.75!(v7)$);
\coordinate (d3)  at  ($(v3)!.625!(v4)$);

\coordinate (c1)  at (intersection of v1--v2 and v5--v6);
\node [inner sep=2pt] (c1n)  at (c1){};
\coordinate (c2)  at (intersection of v2--v3 and v6--v7);
\node [inner sep=2pt] (c2n)  at (c2){};
\coordinate (c3)  at (intersection of v3--v4 and v7--v8);
\node [inner sep=2pt] (c3n)  at (c3){};
\coordinate (v9)  at (5,0);
  
   \draw[->,very thick,blue] (v9)--(v8)--(v7)--(c2n)--(v6)--(v5)--node [above ]{\textcolor{black}{\tiny $r-\epsilon$}}(v4)--(c3n)--(v3)--(v2)--(c1n)--(v1)--(v0);
\draw[scale=1.25,blue,fill=blue] 
(d1) circle (.05) (d2) circle (.05) (d3) circle (.05)
(c1) circle (.03) (c2) circle (.03) (c3) circle (.03);
\end{tikzpicture}
$$
where 
\begin{equation}
B_D=\{s,r,q,p\},
\end{equation}
the global linear order being from left to right.
\begin{lemma}
For any (long knot) diagram, an admissible assignment exists.
\end{lemma}
\begin{proof}
Let $c_D$ be the number of  crossings and $p\in B_D$ the last (or the left infinite) bridge. We start by assigning the unique possible elements to all segments of $p$ and let $p+m\epsilon$ be the element assigned to the first segment of $p$. Then  we  continue assigning elements one by one to the rest of segments  starting from the first segment (of the right infinite bridge) and following the (global) linear order on the set of segments,  always respecting the rules~(\ref{rule1})--(\ref{rule4}). Let $q+n\epsilon$ be the element assigned to the last segment (of bridge $q$) among all the segments excluding the segments of $p$.  In order to verify admissibility of our assignment it suffices to check that $m+n$ is an odd integer. This is indeed the case because this number is obtained from zero by adding $2c_D-1$ odd integers.
\end{proof}
Having fixed an admissible assignment on $D$,  let $\alpha,\beta, \gamma,\delta$  be the elements assigned to four segments around a crossing arranged according to this picture
$$ 
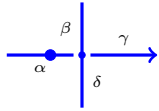
\begin{tikzpicture}[yscale=1]
\coordinate (c)  at (0,0);
\node [inner sep=3pt] (cn)  at (c){};
  \draw[very thick,blue]  (0,-.7)-- node [right]{\textcolor{black}{\tiny $\delta$}}(c)--node [left]{\textcolor{black}{\tiny $\beta$}}(0,.7);
   \draw[->,very thick,blue]  (-1,0)--  node [below]{\textcolor{black}{\tiny $\alpha$}}(cn)-- node [above]{\textcolor{black}{\tiny $\gamma$}}(1,0);
   \draw[scale=1.4,blue,fill=blue] 
(c) circle (.03) (-.3,0) circle (.05) ;
\end{tikzpicture}
 $$
Then, we assign to the auxiliary vertical line that departs from the bridge vertex next to that crossing the element
\begin{equation}
(\alpha+\gamma-\beta -\delta)(-1)^k
\end{equation}
where $k\ge0$ is the number of intersections of that vertical line with the diagram $D$ (recall that the bridge vertex the line starts from is counted or not as an intersection depending on the relative direction, as it was explained above). 

That finishes our combinatorial set-up,
and the following picture illustrates it in the case of the left-handed trefoil diagram:
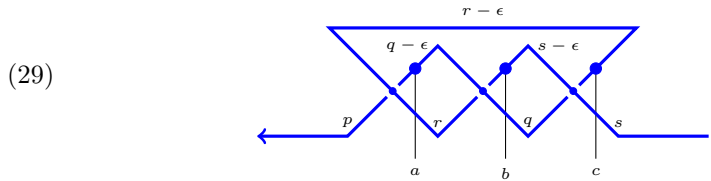
\begin{equation}\label{decd}
\begin{tikzpicture}[baseline=20,scale=1.2]
\coordinate (v0)  at (0,0);
\coordinate [label=above:\textcolor{black}{\tiny $p$}]  (v1) at (1,0);
\coordinate [label=left:\textcolor{black}{\tiny $q-\epsilon$}] (v2)  at (2,1);
\coordinate [label=above:\textcolor{black}{\tiny $q$}] (v3)  at (3,0);
\coordinate (v4)  at (4.2,1.2);
\coordinate   (v5)  at (0.8,1.2);
\coordinate  [label=above:\textcolor{black}{\tiny $r$}] (v6)  at (2,0);
\coordinate  [label=right:\textcolor{black}{\tiny $s-\epsilon$}] (v7)  at (3,1);
\coordinate  [label=above:\textcolor{black}{\tiny $s$}] (v8)  at (4,0);
\coordinate (d1)  at  ($(v1)!.75!(v2)$);
\coordinate (d2)  at  ($(v6)!.75!(v7)$);
\coordinate (d3)  at  ($(v3)!.625!(v4)$);

\coordinate (c1)  at (intersection of v1--v2 and v5--v6);
\node [inner sep=2pt] (c1n)  at (c1){};
\coordinate (c2)  at (intersection of v2--v3 and v6--v7);
\node [inner sep=2pt] (c2n)  at (c2){};
\coordinate (c3)  at (intersection of v3--v4 and v7--v8);
\node [inner sep=2pt] (c3n)  at (c3){};
\coordinate (v9)  at (5,0);
\draw (d1)--+(-90:1) node[below]{\textcolor{black}{\tiny $a$}}(d2)--+(-90:1) node[below]{\textcolor{black}{\tiny $b$}}(d3)--+(-90:1) node[below]{\textcolor{black}{\tiny $c$}};
  
   \draw[->,very thick,blue] (v9)--(v8)--(v7)--(c2n)--(v6)--(v5)--node [above ]{\textcolor{black}{\tiny $r-\epsilon$}}(v4)--(c3n)--(v3)--(v2)--(c1n)--(v1)--(v0);
\draw[scale=1.25,blue,fill=blue] 
(d1) circle (.05) (d2) circle (.05) (d3) circle (.05)
(c1) circle (.03) (c2) circle (.03) (c3) circle (.03);
\end{tikzpicture}
\end{equation}
where the auxiliary elements $a,b,c$ assigned to the vertical lines are defined by the equations 
\begin{align}
\label{al1}&a=(q+p-2r)(-1),\\
 &b=(s+r-2q)(-1),\\
\label{al3} &c=(r+q-2s)(-1),
\end{align}
and no bridge vertex is counted as intersection.
 
Recall that $C_D$ denotes the set of crossings of $D$ and we let $I_D$ denote the set of intersections (of $D$ with the auxiliary vertical lines).
Our main definition is the following quadratic polynomial
\begin{equation}
\mbe{D}:=\sum_{x\in C_D} \me{x}+\sum_{y\in I_D}\ms{y}
\end{equation}
where
\begin{equation}
\me{x}=(\beta-\gamma)\delta \qquad \text{if}\qquad x\ =\ 
\begin{tikzpicture}[baseline=-2,yscale=1]
\coordinate (c)  at (0,0);
\node [inner sep=3pt] (cn)  at (c){};
  \draw[<-,very thick,blue]  (0,-.7)-- node [right]{\textcolor{black}{\tiny $\delta$}}(c)--node [left]{\textcolor{black}{\tiny $\beta$}}(0,.7);
   \draw[->,very thick,blue]  (-1,0)--
   (cn)-- node [above]{\textcolor{black}{\tiny $\gamma$}}(1,0);
   \draw[scale=1.4,blue,fill=blue] 
(c) circle (.03) (-.3,0) circle (.05) ;
\end{tikzpicture}
\end{equation}
\begin{equation}
\me{x}=(\delta-\gamma)\delta \qquad \text{if}\qquad x\ =\ 
\begin{tikzpicture}[baseline=-2,yscale=1]
\coordinate (c)  at (0,0);
\node [inner sep=3pt] (cn)  at (c){};
  \draw[->,very thick,blue]  (0,-.7)-- node [right]{\textcolor{black}{\tiny $\delta$}}(c)--node [left]{\textcolor{black}{\tiny $\beta$}}(0,.7);
   \draw[->,very thick,blue]  (-1,0)--
   (cn)-- node [above]{\textcolor{black}{\tiny $\gamma$}}(1,0);
   \draw[scale=1.4,blue,fill=blue] 
(c) circle (.03) (-.3,0) circle (.05) ;
\end{tikzpicture}
\end{equation}
and
\begin{equation}
\ms{y}=(-1)^{k} (\alpha-\epsilon)a 
\end{equation}
if $y$ is the $k$-th intersection on the vertical line which carries auxiliary element $a$ and $\alpha$ is the element assigned to the segment with which the line intersects at $y$. 

Following the analogy with classical and quantum field theory, we will call this polynomial  the \emph{action functional} of the diagram.
 
As an example of calculation, the action functional of the diagram~\eqref{decd} reads explicitly as follows:
\begin{multline}
\mbe{D}=
(r-p)(r-\epsilon)+(q-r)(q-\epsilon)+(s-q)(s-\epsilon)\\-(r-\epsilon)a-(q-\epsilon)b-(s-\epsilon)c\\
=(r-p-a)(r-\epsilon)+(q-r-b)(q-\epsilon)+(s-q-c)(s-\epsilon)\\
=(q-r)(r-\epsilon)+(s-q)(q-\epsilon)+(r-s)(s-\epsilon)\\
=(q-r)r+(s-q)q+(r-s)s=qr+sq+rs-r^2-q^2-s^2
\end{multline}
where in the third equality we used relations~\eqref{al1}--\eqref{al3}.
\begin{remark}
In this simple example, any dependence on the element $\epsilon$ has dropped out completely from the answer for the action functional, even without using the condition $2\epsilon=1$. That means that if we would put $\epsilon=0$ from the beginning, we would come up with the same answer. This is not always the case in general, at least on the level of the action functional, but nonetheless we cannot exclude the possibility that putting $\epsilon=0$ from the very beginning gives rise to the same final topological result, but the proof of the main theorem to be given in a separate writing relies on the fact that $\epsilon=1/2$.
\end{remark}
 Using the notation $X^Y$ for the set of maps from $Y$ to  $X$, let $\tilde F_D$ be the subset of $\R^{B_D}$ specified by the conditions that all auxiliary elements assigned to the vertical lines are integers. Then,  the image $F_D$  of $\tilde F_D$ in $\T^{B_D}$ with respect to the projection map $\R\to \T$ is known to be identified with $H^1(\tilde X_2,\T)$, where $\tilde X_2$ is the double cyclic covering of the complement of the knot corresponding to $D$, see for example~\cite{MR1959408}. In particular, for any Reidemeister equivalent diagrams $D$ and $D'$ the sets $F_D$ and $F_{D'}$ are in a natural bijection.
\begin{theorem}\label{thm}
 The function $e^{2\pi i \mbe{D}}$ restricted to $\tilde F_D$ factors through a well defined element $e^{2\pi i\mfun{D}}$ of  $\T^{F_D}$ invariant under all Reidemeister moves in the set of all long knot diagrams, and this element is independent of the image of the free generator of  $H_1(\tilde X_2)$.
 \end{theorem}
\def\cprime{$'$} \def\cprime{$'$}


  \end{document}